\documentclass[11pt]{article}
\usepackage[T1]{fontenc}
\usepackage{lmodern}
\usepackage[letterpaper,margin=1in]{geometry}
\usepackage{amsmath,amssymb,amsthm,mathtools}
\usepackage{booktabs,array}
\usepackage{microtype}
\usepackage{needspace}
\usepackage{xurl}
\usepackage[dvipsnames]{xcolor}
\usepackage{tikz}
\usetikzlibrary{calc,positioning,arrows.meta,fit,decorations.pathreplacing}
\usepackage[hidelinks,unicode]{hyperref}
\hypersetup{pdftitle={An upper bound of 60 for the Ramsey number R(3,3,3,3)},pdfauthor={Milos Tatarevic},pdfsubject={Ramsey theory; computer-assisted proof}}

\newtheorem{theorem}{Theorem}[section]
\newtheorem{lemma}[theorem]{Lemma}
\newtheorem{proposition}[theorem]{Proposition}

\newtheorem{computation}[theorem]{Computation}
\theoremstyle{definition}
\newtheorem{definition}[theorem]{Definition}

\theoremstyle{remark}

\newcommand{\U}{\mathrm U}                 
\newcommand{\T}{\mathrm T}                 
\newcommand{\pA}{\mathrm A}                
\newcommand{\pB}{\mathrm B}
\newcommand{\pC}{\mathrm C}
\newcommand{\pD}{\mathrm D}
\newcommand{\pE}{\mathrm E}
\DeclareMathOperator{\Aut}{Aut}            

\usepackage{enumitem}
\setlist{noitemsep,topsep=2pt}
\title{An upper bound of 60 for the Ramsey number\texorpdfstring{\\$R(3,3,3,3)$}{ R(3,3,3,3)}}
\author{Milos Tatarevic\\[3pt]\normalsize Impossible Research\\\normalsize\texttt{milos.tatarevic@gmail.com}}
\date{}

\begin{document}
\maketitle
\begin{abstract}
We improve the upper bound on the classical Ramsey number $R(3,3,3,3)$ from 62 to 60. The proof is based on a global incidence count, classifications of small neighborhoods and exhaustive SAT computations. We formalize the mathematical reductions in Lean and refute every resulting formula with two SAT solvers.
\end{abstract}
\section{Introduction}\label{sec:intro}
The Ramsey number $R(k_1,\ldots,k_q)$ is the least integer $n$ such that every $q$-coloring of the edges of $K_n$ contains a monochromatic $K_{k_i}$ in color $i$ for some $i$. We call a coloring \emph{good} if it has no monochromatic triangle. Fettes, Kramer and Radziszowski proved $R(3,3,3,3)\le62$~\cite{FKR}. As of 2026, the dynamic survey of small Ramsey numbers lists $51\le R(3,3,3,3)\le62$~\cite[Section 6.1.1]{Radziszowski}. We improve the upper bound to 60.

\begin{theorem}\label{thm:main}
Every four-coloring of the edges of $K_{60}$ contains a monochromatic triangle. Consequently, $R(3,3,3,3)\le60$.
\end{theorem}

\subsection{Proof roadmap}\label{sec:roadmap}
Suppose a good four-coloring of $K_{60}$ exists, and choose a root $r$. Its four color neighborhoods $N_i(r)$, the vertices joined to $r$ in color $i$, partition the other 59 vertices. For simplicity, we call the color neighborhoods of a vertex its \emph{groups}. An edge of color $i$ inside $N_i(r)$ would form a monochromatic triangle with $r$, so each group carries a good three-coloring and has at most 16 vertices, as $R(3,3,3)=17$. Their deficits from 16 add up to 5, which leaves six \emph{profiles}:
\[
\begin{array}{c@{\qquad}c@{\qquad}c}
\pA:(11,16,16,16)&\pB:(12,15,16,16)&\pC_1:(13,14,16,16),\\[2pt]
\pD:(13,15,15,16)&\pC_2:(14,14,15,16)&\pE:(14,15,15,15).
\end{array}
\]
Up to relabeling and permuting colors, a good three-coloring on 16 vertices is one of two templates, untwisted ($\U$) and twisted ($\T$), and on 15 vertices a template with one vertex removed. A group with 15 or 16 vertices we call \emph{critical}. A template fixes the edges inside its group but not those between groups.

We say that a vertex $u$ \emph{sends a ticket} to $v$ when the neighborhood of $u$ in the color of $uv$ is critical. This neighborhood contains $v$ and crosses the three groups of $v$ not containing $u$, so its template constrains edges between these groups. We denote the numbers of tickets that $v$ receives and sends by $H(v)$ and $w(v)$. Since a vertex sends tickets exactly along its critical colors, vertices of profiles A, B, C$_1$, D, C$_2$ and E send $48$, $47$, $32$, $46$, $31$ and $45$ tickets. Every ticket has one sender and one receiver, so
\[
\sum_v H(v)=\sum_v w(v).
\]

To prove Theorem~\ref{thm:main}, it would suffice to show that every vertex \emph{falls short}, that is, receives fewer tickets than it sends. The differences $w(v)-H(v)$ would then be positive everywhere, while by double counting they sum to zero.

In Section~\ref{sec:abd}, we show that vertices of profiles A, B, D and E fall short, as they receive at most $47$, $46$, $45$ and $44$ tickets. A root $r$ that violates its bound receives more tickets than it has neighbors outside a largest group. That group therefore contains a sender, and we choose one as the \emph{mark} $p$. Its critical neighborhood in the color of $rp$ contains $r$ and misses the group of $p$. It meets the other three groups of $r$ in \emph{supports}. A support has at most five vertices, since its edges avoid two colors and $R(3,3)=6$. As the supports have 14 or 15 vertices together, each has four or five. The template of the marked neighborhood fixes the colors between them.

We enumerate the placements of the marked neighborhood up to symmetries using the actual neighborhoods. Every placement produces a SAT formula on all 60 vertices, with the edge colors fixed by the templates. SAT solvers refuted all 287,542 formulas. This way, we prove the four ticket bounds. Note that the bounds do not exclude roots of profiles A, B, D and E. We use them only through double counting.

For profiles C$_1$ and C$_2$, we could not prove the analogous bounds $H\le31$ and $H\le30$, respectively, which would finish the proof. If a C$_1$ or C$_2$ vertex receives 32 or 31 tickets, respectively, they may all come from its two critical groups. The critical neighborhood of such a sender crosses the other critical group and the two unclassified noncritical groups, and leaves many edges free.

A ticket from a noncritical group is different. The critical neighborhood of its sender crosses both critical groups, and the same construction works with no template on either noncritical group. We call such a C$_1$ or C$_2$ root and its mark an \emph{entrance}. In Section~\ref{sec:entrance}, we exclude all entrances with 692 C$_1$ cases and 7,840 C$_2$ cases.

Without an entrance, a C$_1$ or C$_2$ vertex receives tickets only from its two critical groups, so $H\le w$. With the bounds at A, B, D and E, all differences $w-H$ are nonnegative, and double counting forces every vertex to \emph{break even}, with $H=w$. Thus, profiles A, B, D and E do not occur. We call a coloring \emph{perfectly balanced} if it has no entrance and every vertex breaks even. In such a coloring, every C$_1$ or C$_2$ vertex receives a ticket from each vertex of its two critical neighborhoods.

Therefore, if $u$ sends $v$ a ticket, then $v$ sends one back. A degree count shows that the color graphs inside a critical neighborhood are connected, so all its vertices have the same critical colors. A second count shows that critical neighbors have the same critical pair. Any two vertices share a critical neighbor, as $31+31>60$, so all have the same two critical colors (Section~\ref{sec:allC}).

We name these two colors 0 and 1, and at a root $r$ put $X=N_0(r)$ and $Y=N_1(r)$. For $x\in X$ and $y\in Y$, the opposite neighborhoods $N_1(x)$ and $N_0(y)$ are critical, so the part of each visible in $X\cup Y$ embeds into $\U$ or $\T$ with all colors preserved. These embedding conditions share the edges between $X$ and $Y$.

For $x\in X$, we impose its condition only when five visible vertices lie in $X$ and span a two-colored pentagon fixed by the template on $X$, and likewise for $y\in Y$. By enumeration, six tables with the pentagon at fixed positions cover all its placements in $\U$ and $\T$. Each embedding then reduces to placing the other visible vertices in one of these tables. We write these conditions into one formula on all 60 vertices, which also fixes the root partition and the templates on $X$ and $Y$. If every vertex is C$_1$, then up to exchanging colors 0 and 1, three all-five formulas remain, for UU, UT and TT. Otherwise, four remain at a C$_2$ root. All seven are unsatisfiable, so no coloring is perfectly balanced (Sections~\ref{sec:guard} to~\ref{sec:overlap}).

In Section~\ref{sec:verification}, we describe the computations and the formal verification, and Figure~\ref{fig:proofmap} shows the proof dependencies.

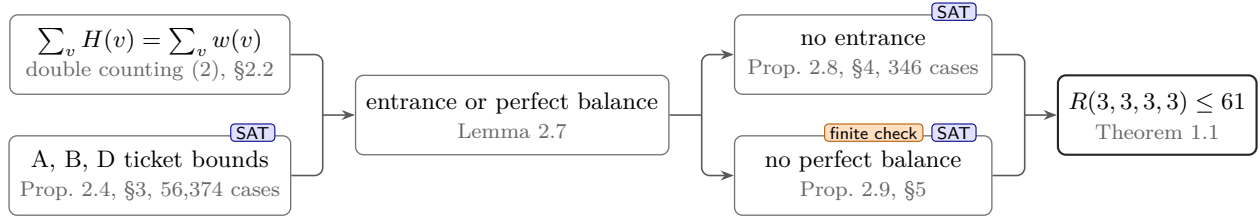
\begin{figure}[htbp]
\centering
\begingroup\hyphenpenalty=10000\exhyphenpenalty=10000
\def\gr#1{{\scriptsize\color{black!60}#1}}
\begin{tikzpicture}[scale=.98,transform shape,
  box/.style={draw=black!55,line width=.5pt,rounded corners=3pt,font=\footnotesize,
    align=center,inner xsep=4pt,inner ysep=3pt,minimum height=1.05cm},
  tag/.style={font=\sffamily\fontsize{6.5}{7}\selectfont,inner xsep=2pt,inner ysep=1pt,rounded corners=1.5pt,line width=.4pt},
  sat/.style={tag,draw=blue!55!black,fill=blue!12},
  fin/.style={tag,draw=orange!70!black,fill=orange!25},
  arr/.style={-{Stealth[length=5pt]},black!65,line width=.6pt},
  ln/.style={black!65,line width=.6pt,rounded corners=3pt}]
  \node[box,minimum width=3.7cm] (id) at (0,.8)
    {$\sum_v H(v)=\sum_v w(v)$\\\gr{double counting~(\ref{eq:identity}), \S\ref{sec:tickets}}};
  \node[box,minimum width=3.7cm] (abd) at (0,-.8)
    {A, B, D, E ticket bounds\\\gr{Prop.~\ref{prop:abd}, \S\ref{sec:abd}, 287,542 cases}};
  \node[box,anchor=west] (lem) at ($(id.east|-0,0)+(.85,0)$)
    {entrance or perfect balance\\\gr{Lemma~\ref{lem:entrance}}};
  \node[box,anchor=west,minimum width=3.4cm] (ent) at ($(lem.east|-id)+(.85,0)$)
    {no entrance\\\gr{Prop.~\ref{prop:noentrance}, \S\ref{sec:entrance}, 8,532 cases}};
  \node[box,anchor=west,minimum width=3.4cm] (bal) at ($(lem.east|-abd)+(.85,0)$)
    {no perfect balance\\\gr{Prop.~\ref{prop:nobalance}, \S\ref{sec:balance}}};
  \node[box,anchor=west,draw=black!85,line width=.8pt] (thm) at ($(ent.east|-0,0)+(.85,0)$)
    {$R(3,3,3,3)\le60$\\\gr{Theorem~\ref{thm:main}}};
  \node[sat,anchor=south east] at ([xshift=-6pt,yshift=-2.5pt]abd.north east) {SAT};
  \node[sat,anchor=south east] at ([xshift=-6pt,yshift=-2.5pt]ent.north east) {SAT};
  \node[sat,anchor=south east] (t1) at ([xshift=-6pt,yshift=-2.5pt]bal.north east) {SAT};
  \node[fin,anchor=south east] at ([xshift=-3pt]t1.south west) {finite check};
  \coordinate (j1) at ($(id.east|-0,0)!.5!(lem.west)$);
  \draw[ln] (id.east) -| (j1);
  \draw[ln] (abd.east) -| (j1);
  \draw[arr] (j1) -- (lem.west);
  \coordinate (j2) at ($(lem.east)!.5!(lem.east-|ent.west)$);
  \draw[ln] (lem.east) -- (j2);
  \draw[arr,rounded corners=3pt] (j2) |- (ent.west);
  \draw[arr,rounded corners=3pt] (j2) |- (bal.west);
  \coordinate (j3) at ($(ent.east|-0,0)!.5!(thm.west)$);
  \draw[ln] (ent.east) -| (j3);
  \draw[ln] (bal.east) -| (j3);
  \draw[arr] (j3) -- (thm.west);
\end{tikzpicture}
\endgroup
\caption{Proof dependencies. Double counting and the A, B, D, E bounds give Lemma~\ref{lem:entrance}, which leaves two cases, and excluding both proves the theorem. Tags mark the steps that rest on SAT exclusions or on a finite check of template embeddings, and the other steps are deductions.}
\label{fig:proofmap}
\end{figure}

\section{Preliminaries}\label{sec:prelim}

\subsection{Colorings, degrees and templates}\label{sec:templates}
To reach a contradiction, let $c$ be a good four-coloring of $K_{60}$. We denote the color of an edge $uv$ by $c(uv)$. The \emph{color-$i$ neighborhood} of a vertex $v$ is $N_i(v)=\{u\ne v:c(uv)=i\}$, and $d_i(v)=\lvert N_i(v)\rvert$ is its \emph{color-$i$ degree}. We also call the four color neighborhoods of $v$ its \emph{groups}. We use $R(3,3)=6$ and $R(3,3,3)=17$~\cite{GreenwoodGleason}. A good two-coloring of $K_5$ consists of two complementary five-cycles.

An edge of color $i$ inside $N_i(v)$ would form a color-$i$ triangle with $v$, so $N_i(v)$ carries a good three-coloring. Since $R(3,3,3)=17$ and the four degrees add up to 59, we have
\begin{equation}\label{eq:degrees}
11\le d_i(v)\le16,\qquad \sum_{i=0}^{3}d_i(v)=59.
\end{equation}
We call color $i$ \emph{critical} at $v$ if $d_i(v)\ge15$, and $N_i(v)$ is then a \emph{critical neighborhood}.

\begin{theorem}[Kalbfleisch and Stanton~\cite{KalbfleischStanton}, and Heinrich~\cite{Heinrich}]\label{thm:classification}
Up to relabeling the vertices and permuting the three colors, there are exactly two good three-colorings of $K_{16}$, the untwisted template $\U$ and the twisted template $\T$. Up to the same operations, every good three-coloring of $K_{15}$ is obtained from $\U$ or $\T$ by removing a vertex.
\end{theorem}

We label the vertices of each template \(0,\ldots,15\), calling them positions. We label \(U^{-}\) and \(T^{-}\) the colorings obtained by deleting position 15 from $\U$ and $\T$. The exact colored adjacency matrices for all templates are supplied in our repository \cite{repository}.

The \emph{automorphisms} of a template are the permutations of its positions that preserve all edge colors. They form the groups $\Aut(\U)$ and $\Aut(\T)$.

\begin{lemma}\label{lem:templatesym}
\begin{enumerate}
\item[(a)] The groups $\Aut(\U)$ and $\Aut(\T)$ act transitively on the 16 positions.
\item[(b)] $\lvert\Aut(\U)\rvert=160$ and $\lvert\Aut(\T)\rvert=32$, so the stabilizer of a position has order 10 in $\Aut(\U)$ and order 2 in $\Aut(\T)$. The automorphisms of $\U^-$ and of $\T^-$ are exactly the automorphisms of $\U$ and of $\T$ that fix position 15.
\item[(c)] Let $\tau$ be a permutation of the three colors and $z$ a position. There is a permutation of the positions that fixes $z$ and carries every edge of $\U$ of color $j$ to an edge of color $\tau(j)$, for $j=0,1,2$. The same holds for $\T$.
\end{enumerate}
\end{lemma}

\begin{proof}
We find the automorphisms of the four templates by backtracking. A position is mapped only to an unused position with the same color degrees, and only if its edge colors to the positions mapped so far agree. The resulting lists give (a) and (b). The same search, matching color $j$ to color $\tau(j)$, finds a permutation realizing each $\tau$~\cite{repository}.

For (c), if $\sigma$ realizes $\tau$ (maps color-$j$ edges to color-$\tau(j)$ edges), so does $\alpha\circ\sigma$ for every automorphism $\alpha$, and by (a) some $\alpha$ has $\alpha(\sigma(z))=z$.
\end{proof}

When we place a template on $N_i(v)$, its colors $0,1,2$ stand for the other three colors in increasing order. Theorem~\ref{thm:classification} and Lemma~\ref{lem:templatesym}(c) let us fix this \emph{palette} by relabeling the template vertices. When $d_i(v)=15$, we take position 15 as the missing one and realize the palette by a permutation fixing it.

Thus, every critical $N_i(v)$ is a copy of $\U$, $\T$, $\U^-$ or $\T^-$ with this palette, and its \emph{type} is $\U$ or $\T$. The types stay distinct after deletion, since $\U^-$ and $\T^-$ have 10 and 2 automorphisms by Lemma~\ref{lem:templatesym}(b). On a 16-vertex neighborhood, transitivity lets us put any vertex at position 15.

\subsection{Profiles and tickets}\label{sec:tickets}
The deficits $16-d_i(v)$ are nonnegative and sum to 5, so the six partitions of 5 into at most four parts give the sorted degree profiles in Table~\ref{tab:profiles}.

\begin{table}[htbp]
\centering
\caption{Root profiles and the number of tickets sent.}\label{tab:profiles}
\begin{tabular}{@{}cccc@{}}
\toprule
Profile & Sorted color degrees & Critical neighborhoods & $w(v)$\\
\midrule
A & $(11,16,16,16)$ & 3 & 48\\
B & $(12,15,16,16)$ & 3 & 47\\
C$_1$ & $(13,14,16,16)$ & 2 & 32\\
D & $(13,15,15,16)$ & 3 & 46\\
C$_2$ & $(14,14,15,16)$ & 2 & 31\\
E & $(14,15,15,15)$ & 3 & 45\\
\bottomrule
\end{tabular}
\end{table}

\begin{definition}\label{def:ticket}
A vertex $u$ \emph{sends a ticket} to $v\ne u$ when $d_{c(uv)}(u)\ge15$. We denote the numbers of tickets that $v$ receives and sends by $H(v)$ and $w(v)$. Thus
\[
H(v)=\bigl|\{u\ne v:d_{c(uv)}(u)\ge15\}\bigr|,
\qquad
w(v)=\sum_{i:d_i(v)\ge15}d_i(v).
\]
A vertex \emph{breaks even} if $H(v)=w(v)$ and \emph{falls short} if $H(v)<w(v)$.
\end{definition}

Double counting the tickets, once at the sender and once at the receiver, gives
\begin{equation}\label{eq:identity}
\sum_v H(v)=\sum_v w(v).
\end{equation}

\subsection{The two branches}\label{sec:global}
In Section~\ref{sec:abd}, we prove the following bounds.
\begin{proposition}[The A, B, D and E bounds]\label{prop:abd}
In a good four-coloring of $K_{60}$, a vertex of profile A, B, D or E receives at most $47$, $46$, $45$ or $44$ tickets, respectively.
\end{proposition}
Since each bound is one below the number sent in Table~\ref{tab:profiles}, these vertices fall short. Section~\ref{sec:abd} gives no such bound for profiles C$_1$ and C$_2$, so the proof splits into two branches.

\begin{definition}[Entrance]\label{def:entrance}
An \emph{entrance} is a pair $(r,p)$ such that $r$ has profile C$_1$ or C$_2$, $p$ belongs to a noncritical neighborhood of $r$, and $p$ sends $r$ a ticket. We call $r$ the \emph{root} and $p$ the \emph{mark}.
\end{definition}
\begin{definition}[Perfect balance]\label{def:balance}
A good coloring of $K_{60}$ is \emph{perfectly balanced} if it has no entrance and every vertex breaks even.
\end{definition}

\begin{lemma}[Entrance or perfect balance]\label{lem:entrance}
Assuming the bounds of Proposition~\ref{prop:abd}, exactly one of the following holds:
\begin{enumerate}
\item An entrance exists.
\item The coloring is perfectly balanced. Only profiles C$_1$ and C$_2$ occur. Every vertex receives tickets from exactly its neighbors in critical colors.
\end{enumerate}
\end{lemma}
\begin{proof}
Without an entrance, a C$_1$ or C$_2$ vertex receives tickets only from its two critical neighborhoods, so $H(v)\le w(v)$. The bounds give $H(v)<w(v)$ at A, B, D and E. The differences $w(v)-H(v)$ are thus nonnegative, and~\eqref{eq:identity} directly implies that they all vanish. Consequently, the coloring is perfectly balanced, no vertex has profile A, B, D or E, and every C$_1$ or C$_2$ vertex receives all the tickets listed in item~2. By definition, the items cannot both hold.
\end{proof}

In Sections~\ref{sec:entrance} and~\ref{sec:balance}, we exclude the two branches.
\begin{proposition}[No entrance]\label{prop:noentrance}
No good four-coloring of $K_{60}$ has an entrance.
\end{proposition}
\begin{proposition}[No perfect balance]\label{prop:nobalance}
No good four-coloring of $K_{60}$ is perfectly balanced.
\end{proposition}
\begin{proof}[Proof of Theorem~\ref{thm:main} from the three propositions]
By Proposition~\ref{prop:abd} and Lemma~\ref{lem:entrance}, a good four-coloring of $K_{60}$ has an entrance or is perfectly balanced. Propositions~\ref{prop:noentrance} and~\ref{prop:nobalance} exclude both.
\end{proof}

\subsection{Symmetries, formulas and coverage}\label{sec:formulas}
A \emph{symmetry} $(\sigma,\tau)$ consists of a vertex permutation $\sigma$ and a global color permutation $\tau$. It sends $c$ to $c'$ with $c'(\sigma(u)\sigma(v))=\tau(c(uv))$ and preserves goodness, profiles and tickets. We extend permutations of vertex subsets by the identity.

\begin{definition}\label{def:case}
A \emph{marked case} $\kappa$ is a finite list of signed edge facts on the vertices $0,\ldots,59$, each stating that an edge $uv$ has, or does not have, a given color $i$. Its formula $\Phi_\kappa$ uses variables $x_{e,i}$ for the $1{,}770$ edges $e$ and four colors $i$. Besides the unit clauses for those facts, it contains
\begin{equation}\label{eq:baseclauses}
\begin{aligned}
&\bigvee_{i=0}^3 x_{e,i} &&\text{for each edge }e,\\
&\neg x_{e,i}\vee\neg x_{e,j} &&\text{for each edge }e\text{ and }i<j,\\
&\neg x_{uv,i}\vee\neg x_{uw,i}\vee\neg x_{vw,i}
&&\text{for each triangle }uvw\text{ and color }i.
\end{aligned}
\end{equation}
There are $7{,}080$ variables and $149{,}270$ shared clauses. Only the unit clauses depend on $\kappa$.
\end{definition}

\begin{lemma}[Encoding and coverage]\label{lem:soundness}\label{cor:cover}
The satisfying assignments of $\Phi_\kappa$ correspond exactly to good four-colorings satisfying the facts of $\kappa$. A list of cases \emph{covers} a class of good colorings if every coloring in that class satisfies one of its cases after some symmetry. If every formula in a cover is unsatisfiable, no coloring in the covered class exists.
\end{lemma}
\begin{proof}
The first two clause families give each edge exactly one color, the third forbids monochromatic triangles, and the unit clauses impose exactly the listed facts. A coloring in the class would, after a symmetry, satisfy one of the cases and hence its formula.
\end{proof}

\section{Marked neighborhoods and the A, B, D and E bounds}\label{sec:abd}

Here and in Section~\ref{sec:entrance}, we place a sender's critical neighborhood across the receiver's groups.

\subsection{A common construction}\label{sec:marked-construction}
Let $r$ be a root with groups $R_i=N_i(r)$ such that the \emph{foreign groups} $R_1$ and $R_2$ carry placed critical templates. Suppose $p\in R_3$ sends $r$ a ticket, that is, $d_3(p)\in\{15,16\}$. We call $p$ the \emph{mark}, $S=N_3(p)$ the \emph{brace} and $S_i=S\cap R_i$, for $i=0,1,2$, the \emph{supports} (Figure~\ref{fig:marked-supports}). We assume no template on $R_0$ and call it \emph{free}. For fixed sizes of the root's groups, all cases use the same vertices for the root, the mark and the four groups.

\begin{lemma}\label{lem:marked-shapes}
The brace is $\{r\}\cup S_0\cup S_1\cup S_2$. Its ordered support sizes $(\lvert S_0\rvert,\lvert S_1\rvert,\lvert S_2\rvert)$ have one of the four shapes
\[
(4,5,5),\qquad(5,4,5),\qquad(5,5,4),\qquad(5,5,5).
\]
A five-support carries complementary five-cycles, and a four-support carries complementary paths.
\end{lemma}
\begin{proof}
The root lies in $S$, and no vertex $v\ne p$ of $R_3$ does, since $rpv$ would be a color-3 triangle. Edges inside $S_i$ avoid colors $i$ and 3, so $|S_i|\le5$ by $R(3,3)=6$. As the sizes add up to $|S|-1\in\{14,15\}$, only the four shapes remain. By classification, $S$ is a copy of $\U$, $\T$, $\U^-$ or $\T^-$. In a good three-coloring of $K_{16}$, the three color degrees are at most five by $R(3,3)=6$ and add up to 15, so each is five. Hence, the color neighborhoods of $r$ in $\U$ or $\T$ carry good two-colorings of $K_5$, that is, complementary five-cycles. The supports are these sets, and deleting a position shrinks at most one of them to four vertices, which carry complementary paths.
\end{proof}
\begin{figure}[htbp]
\centering
\definecolor{braceblue}{HTML}{0072B2}
\definecolor{braceorange}{HTML}{E69F00}
\definecolor{bracegreen}{HTML}{009E73}
\definecolor{bracepurple}{HTML}{CC79A7}
\begin{tikzpicture}[
  c0/.style={draw=braceblue,line width=.9pt},
  c1/.style={draw=braceorange,line width=.9pt,dashed},
  c2/.style={draw=bracegreen,line width=1.1pt,densely dotted},
  c3/.style={draw=bracepurple,line width=.9pt,dash dot},
  vtx/.style={circle,fill=black,inner sep=0pt,minimum size=3.6pt},
  nbhd/.style={draw=black!55,rounded corners=7pt,line width=.7pt},
  part/.style={draw=black!40,fill=black!6,rounded corners=4pt,line width=.5pt},
  lbl/.style={font=\small},
  slbl/.style={font=\small}]
  \path[fill=bracepurple!10,draw=bracepurple!85,dashed,line width=.6pt,rounded corners=14pt]
    (-2.85,2.5) -- (2.85,2.5) -- (2.85,-2.5) -- (1.35,-2.5) -- (-.6,-.5) -- (-2.85,1.0) -- cycle;
  \node[slbl,text=bracepurple!70!black] at (0,2.9) {brace $S=N_3(p)$};
  \draw[nbhd] (-6.0,1.0) rectangle (-1.0,2.6);
  \draw[nbhd] (1.0,1.0) rectangle (6.0,2.6);
  \draw[nbhd] (-5.2,-2.6) rectangle (-1.2,-1.0);
  \draw[nbhd] (1.2,-2.6) rectangle (5.2,-1.0);
  \node[lbl,above] at (-3.9,2.6) {$R_1=N_1(r)$: 16 vertices};
  \node[lbl,above] at (3.9,2.6) {$R_2=N_2(r)$: 16 vertices};
  \node[lbl,below] at (-3.2,-2.6) {$R_3=N_3(r)$: 14 vertices};
  \node[lbl,below] at (3.2,-2.6) {$R_0=N_0(r)$: 13 vertices};
  \foreach \k in {0,...,5} {
    \node[vtx] at ({-5.6+.5*\k},2.15) {};
    \node[vtx] at ({5.6-.5*\k},2.15) {};
  }
  \foreach \k in {0,...,4} {
    \node[vtx] at ({-5.35+.5*\k},1.45) {};
    \node[vtx] at ({5.35-.5*\k},1.45) {};
  }
  \foreach \k in {0,...,6} \node[vtx] at ({-4.9+.38*\k},-1.4) {};
  \foreach \k in {0,...,5} \node[vtx] at ({-4.71+.38*\k},-2.2) {};
  \foreach \k in {0,...,4} \node[vtx] at ({3.05+.45*\k},-1.4) {};
  \foreach \k in {0,...,2} \node[vtx] at ({3.275+.45*\k},-2.2) {};
  \foreach \nm/\cx/\cy in {A/-2.0/1.8,B/2.0/1.8,C/2.1/-1.8} {
    \draw[part] ({\cx-.6},{\cy-.6}) rectangle ({\cx+.6},{\cy+.6});
    \foreach \k in {0,...,4} \node[vtx] (\nm\k) at ($(\cx,\cy)+(90+72*\k:.38)$) {};
  }
  \node[slbl,anchor=east] at (-2.62,1.8) {$S_1$};
  \node[slbl,anchor=west] at (2.62,1.8) {$S_2$};
  \node[slbl,anchor=west] at (2.72,-1.8) {$S_0$};
  \draw[black!40,line width=.5pt] (A0) -- (B0);
  \draw[black!40,line width=.5pt] (A4) -- (B1);
  \draw[black!40,line width=.5pt] (A3) -- (B2);
  \node[vtx,minimum size=5pt,label={[lbl]right:$r$}] (r) at (0,0) {};
  \draw[c1] (r) -- (A2);
  \draw[c2] (r) -- (B3);
  \draw[c0] (r) -- (C1);
  \node[circle,draw=black,fill=white,inner sep=0pt,minimum size=6.5pt,line width=.9pt] (p) at (-1.8,-1.5) {};
  \node[lbl,below=1pt of p] {$p$};
  \draw[c3,-{Stealth[length=7pt]}] (p) -- (r);
  \draw[c3] (p) -- (A1);
  \draw[c3] (p) -- (C2);
  \draw[c3] (p) to[bend right=22] (B3);
\end{tikzpicture}
\caption{The construction for a C$_1$ root and shape $(5,5,5)$. The white mark $p\in R_3$ sends $r$ a ticket along a color-3 edge (arrow). Its shaded brace $S=N_3(p)$ contains $r$, avoids $R_3$ and meets the other groups in the supports. Here only $S$, $R_1$ and $R_2$ carry templates. We draw only some edges. The three gray edges between supports lie in the brace and are drawn without colors.}
\label{fig:marked-supports}
\end{figure}
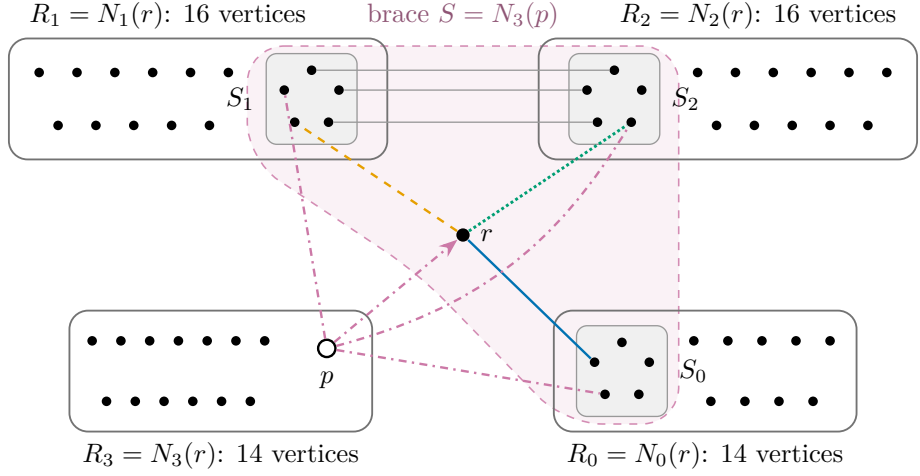

The \emph{reference coloring} of support $i$ uses the two colors of $\{0,1,2\}\setminus\{i\}$. It puts the smaller one on the path $0,1,2,3$ if $|S_i|=4$, or on the cycle $0,1,2,3,4,0$ if $|S_i|=5$, and the larger on the other pairs. A \emph{labeling} of $S_i$, or of any set colored like it, is a color-preserving bijection from the reference vertices onto it.

\begin{definition}\label{def:marked-word}
Let $a=(a_0,a_1,a_2)$ be a shape and $n=1+a_0+a_1+a_2$. Let $\rho\in\{\U,\T\}$ if $n=16$, and $\rho\in\{\U^-,\T^-\}$ if $n=15$. A \emph{base vertex} of $\rho$, which stands for the root, is position 15 if $n=16$, and any position with $a_i$ neighbors of color $i$, for $i=0,1,2$, if $n=15$. For any labelings of the color-$i$ neighbors of a base vertex by the reference coloring of support $i$, the colors between the three sets, read in label order, form a \emph{word} of shape $a$. Its length is 65, or 75 for shape $(5,5,5)$.
\end{definition}

\begin{computation}\label{comp:marked-words}\label{comp:words}
For each shape, these base vertices and labelings give exactly 600 distinct words, 100 from $\U$ or $\U^-$ and 500 from $\T$ or $\T^-$.
\end{computation}

Every brace gives one of these words, for any labelings of its supports. By classification, we place a template on $S$ with $r$ at a base vertex, using transitivity when $|S|=16$, and compose the support labelings with this placement.

We normalize the supports by symmetry. For $i=1,2$, we list the subsets of the template on $R_i$ that admit the reference coloring of $S_i$. From each orbit under the automorphisms of the whole template, we keep one, with a fixed labeling. We place the free support $S_0$ on fixed vertices of $R_0$, in reference order.

For fixed sizes of the root's groups and fixed template types, a \emph{normalized configuration} is a selected pair of foreign supports with one of the 600 words of a matching shape. The automorphism group of the free reference path or cycle has order 2 or 10. With the set stabilizers of the foreign supports, it acts on these words, so we may keep only orbit representatives.

Each normalized configuration fixes:
\begin{enumerate}
\item all root edges;
\item the templates on $R_1$ and $R_2$, and on $R_3$ when present;
\item the internal coloring of the brace and its exact membership as $N_3(p)$.
\end{enumerate}
These specifications constrain one coloring, so they must agree on every shared edge.

For a symmetry $(\sigma,\tau)$, the \emph{pullback} of a fact $c'(uv)=i$ is the fact $c(\sigma^{-1}(u)\sigma^{-1}(v))=\tau^{-1}(i)$, and likewise with $\ne$ for a negative fact. By the definition of a symmetry (Section~\ref{sec:formulas}), $c'$ satisfies a fact exactly when $c$ satisfies its pullback.

We call the list of cases left after these reductions a \emph{retained case list}. It satisfies the \emph{finite coverage condition} if each normalized configuration has a retained case and a symmetry that pulls every fact of the case back to a fact of the configuration. The same check justifies the reduction to orbit representatives and any exchange of equal-sized groups.

\begin{lemma}[Marked-neighborhood cover]\label{lem:marked-cover}
For fixed sizes of the root's groups and placed template types, the normalized configurations cover every good coloring with a mark as above, and so does every retained case list that satisfies the finite coverage condition. Neither claim needs a template on $R_3$.
\end{lemma}
\begin{proof}
By Lemma~\ref{lem:marked-shapes}, each support admits its reference coloring. Hence, for $i=1,2$, an automorphism of the whole template on $R_i$ carries $S_i$ to its chosen orbit representative. A reference labeling of $S_0$, sent in order to the fixed free-support vertices, extends to a permutation of $R_0$, as the complements have equal size. With the identity on $r$ and $R_3$, these maps form a permutation $\sigma$ of all 60 vertices. It maps each group of the root to itself and acts by automorphisms on $R_1$ and $R_2$, so it preserves the root-edge colors, the foreign templates and any template on $R_3$. Let $c'$ be the image of $c$ under the symmetry $(\sigma,\mathrm{id})$.

Since $\sigma$ fixes $p$, for every $v\ne p$ we have
\[
c'(p\sigma(v))=3\quad\Longleftrightarrow\quad c(pv)=3
\quad\Longleftrightarrow\quad v\in S,
\]
which implies that the brace of $p$ in $c'$ is $\sigma(S)$, the root with the selected supports. So $c'$ satisfies the required brace membership. The reference coloring on the free support holds, since the labeling of $S_0$ preserves colors. As shown after Computation~\ref{comp:marked-words}, the colors between the selected supports in $c'$, read in their fixed labelings, form one of the 600 words. With the selected supports, this word gives a normalized configuration whose facts all hold in $c'$.

Finally, each fact of the retained case pulls back to a fact of the configuration, which holds in $c'$. So the symmetry from the finite coverage condition carries $c'$ to a good coloring that satisfies the retained case.
\end{proof}

\subsection{The A, B, D and E cases}\label{sec:marked-abd}\label{sec:marked-e}
Suppose a root violates its bound in Proposition~\ref{prop:abd}. We name the colors so that $R_0,\ldots,R_3$ have sizes $(11,16,16,16)$, $(12,15,16,16)$, $(13,15,15,16)$ or $(14,15,15,15)$. A global color permutation, followed by renormalizing the templates to the fixed palette, exchanges two groups of equal size. So we may order the template types by $\U<\T$ on all three 16-vertex groups for A, on $R_2,R_3$ for B, on $R_1,R_2$ for D and on all three 15-vertex groups for E. The root receives at least 48, 47, 46 or 45 tickets, but only 43 of its neighbors lie outside a given 16-vertex group, and 44 outside a given 15-vertex group. Every largest group, including $R_3$, therefore contains at least five, four or three senders for A, B or D, and at least one for E. We choose one in $R_3$ as the mark $p$. For A, B and D, we place the template of $R_3$ with $p$ at position 15. For E, we move $p$ by an automorphism of $\U^-$ or $\T^-$ to one of three or nine orbit representatives. We then apply Lemma~\ref{lem:marked-cover}.

For D and E, we also use the exchange of $R_1$ with $R_2$ and color 1 with color 2. We then restore the template of $R_3$ by a permutation of positions fixing position 15 and realizing the color exchange. Relabeling the reference path or cycle of the free support restores its coloring. For E, we again move the mark to an orbit representative. When the two foreign types agree, the exchange identifies the shapes $(5,4,5)$ and $(5,5,4)$. For foreign types TT and shapes $(4,5,5)$ and $(5,5,5)$, we also order the orbits of the two five-vertex foreign supports. If the two orbits coincide, we make no further exchange reduction. Table~\ref{tab:marked-abd} lists the retained cases for Lemma~\ref{lem:marked-cover}.

\begin{table}[htbp]
\centering
\small
\caption{Retained A/B/D/E cases, with the foreign types of $R_1,R_2$. For E, the counts include the three or nine mark positions. Counts are per listed marked type, so a row with marked types U,T counts twice. A dash means the exchange of the foreign groups covers shape $(5,5,4)$ through shape $(5,4,5)$.}\label{tab:marked-abd}\label{tab:marked-e}\label{tab:ab}\label{tab:d}
\begin{tabular}{@{}lllrrrrr@{}}
\toprule
Profile & Foreign & Marked & $(4,5,5)$ & $(5,4,5)$ & $(5,5,4)$ & $(5,5,5)$ & Total\\
\midrule
A & UU & U,T & 4 & 4 & 4 & 2 & 14\\
A & UT & T & 16 & 16 & 20 & 4 & 56\\
A & TT & T & 76 & 92 & 92 & 16 & 276\\
B & UU & U,T & 36 & 40 & 36 & 10 & 122\\
B & UT & T & 168 & 184 & 204 & 36 & 592\\
B & TU & U,T & 168 & 220 & 168 & 36 & 592\\
B & TT & T & 828 & 1{,}084 & 996 & 168 & 3{,}076\\
D & UU & U,T & 189 & 408 & -- & 82 & 679\\
D & UT & U,T & 1{,}824 & 1{,}992 & 2{,}388 & 372 & 6{,}576\\
D & TT & U,T & 5{,}256 & 11{,}892 & -- & 1{,}056 & 18{,}204\\
E & UU & U & 567 & 1{,}224 & -- & 246 & 2{,}037\\
E & UU & T & 1{,}701 & 3{,}672 & -- & 738 & 6{,}111\\
E & UT & T & 16{,}416 & 17{,}928 & 21{,}492 & 3{,}348 & 59{,}184\\
E & TT & T & 47{,}304 & 107{,}028 & -- & 9{,}504 & 163{,}836\\
\bottomrule
\end{tabular}
\end{table}

\begin{computation}\label{comp:marked-census}\label{comp:abdcover}
For every A, B, D and E family in Table~\ref{tab:marked-abd}, the retained list satisfies the finite coverage condition. The lists contain 360 A cases, 5,096 B cases, 50,918 D cases and 231,168 E cases, 287,542 in total.
\end{computation}

Each case gives a formula as in Definition~\ref{def:case}. The case lists, symmetry witnesses and formula comparisons accompany the computations of Section~\ref{sec:verification}.

\begin{computation}[The A/B/D/E exclusions]\label{comp:abdsat}\label{comp:esat}
All 287,542 A/B/D/E case formulas are unsatisfiable. Kissat~4.0.4~\cite{Kissat} and CaDiCaL~3.0.1~\cite{Cadical}, run as described in Section~\ref{sec:verification}, refuted each of them.
\end{computation}

\begin{proof}[Proof of Proposition~\ref{prop:abd}]
A vertex violating one of the bounds gives a mark as above. Lemma~\ref{lem:marked-cover} and Computation~\ref{comp:marked-census} directly imply that, after a symmetry, the coloring satisfies a case of Table~\ref{tab:marked-abd}. By Lemma~\ref{lem:soundness}, its edge colors satisfy the formula of that case, which contradicts Computation~\ref{comp:abdsat}.
\end{proof}
\section{Excluding C entrances}\label{sec:entrance}

Let $(r,p)$ be an entrance (Definition~\ref{def:entrance}). We name the colors so that $p$ lies in $R_3$ and $R_1,R_2$ are the critical groups. The ticket from $p$ to $r$ means $d_3(p)\in\{15,16\}$, so the construction of Section~\ref{sec:marked-construction} applies with templates on $R_1,R_2$ and none on $R_0,R_3$.

For C$_1$, both foreign groups have 16 vertices. If necessary, we exchange colors 1 and 2, and with them $R_1$ and $R_2$, to order the foreign types as UU, UT or TT. The actions on supports and words involve only the two foreign templates and the free support. Every vertex $v$ of $R_0$ outside $S_0$ has $c(pv)\ne3$, so enlarging $R_0$ from 11 to 13 or 14 vertices leaves these actions unchanged. Thus, they are the actions of the A row of Table~\ref{tab:marked-abd} with the same foreign types. We keep all four shapes and list the cases in Table~\ref{tab:marked-c}. Each listed case gives two C$_1$ cases, as $R_3$ may have 13 or 14 vertices.

For C$_2$, both noncritical groups have 14 vertices. We complete the 15-vertex foreign template by adding its missing position outside its support, and proceed as for C$_1$. Ordering the foreign types may exchange the two templates, so we delete one position outside the supports in one of the two templates. Two foreign five-supports leave $11+11=22$ choices, and a foreign four-support and five-support leave $12+11=23$.

\begin{table}[htbp]
\centering
\caption{Retained C entrance cases. A shape lists $|S_0|,|S_1|,|S_2|$, with $S_0$ the free support.}\label{tab:marked-c}\label{tab:c}
\begin{tabular}{@{}lrrrrr@{}}
\toprule
Foreign types & $(4,5,5)$ & $(5,4,5)$ & $(5,5,4)$ & $(5,5,5)$ & Total\\
\midrule
UU & 4 & 4 & 4 & 2 & 14\\
UT & 16 & 16 & 20 & 4 & 56\\
TT & 76 & 92 & 92 & 16 & 276\\
\midrule
Total & 96 & 112 & 116 & 22 & 346\\
\midrule
C$_1$ cases & 192 & 224 & 232 & 44 & 692\\
C$_2$ cases & 2{,}112 & 2{,}576 & 2{,}668 & 484 & 7{,}840\\
\bottomrule
\end{tabular}
\end{table}

\begin{proposition}\label{prop:marked-c-cover}
Every good coloring with an entrance satisfies, after a symmetry, all facts of one of the 692 C$_1$ or 7,840 C$_2$ cases.
\end{proposition}
\begin{proof}
The entrance gives a mark as in Lemma~\ref{lem:marked-cover}, which needs no template on $R_3$. For C$_1$, we normalize the two foreign templates and order their types as above. The actions of the A rows fix the mark and preserve both foreign templates. With these actions, the cases of Table~\ref{tab:marked-c} satisfy the finite coverage condition for every shape and type pair, so Lemma~\ref{lem:marked-cover} applies for both sizes of $R_3$.

For C$_2$, we normalize the completed foreign templates in the same way and track the image of the missing position. This image lies outside the supports, so it is one of the deletion choices. We delete it with its incident facts, and every remaining fact pulls back to a fact that holds in the coloring.
\end{proof}

\begin{computation}[The C entrance exclusions]\label{comp:csat}
All 692 C$_1$ and 7,840 C$_2$ case formulas of Table~\ref{tab:marked-c} are unsatisfiable. Kissat and CaDiCaL, run as in Computation~\ref{comp:abdsat}, refuted each of them.
\end{computation}

\begin{proof}[Proof of Proposition~\ref{prop:noentrance}]
By Proposition~\ref{prop:marked-c-cover}, a good coloring with an entrance satisfies, after a symmetry, a case of Table~\ref{tab:marked-c}. By Lemma~\ref{lem:soundness}, its edge colors satisfy the formula of that case, contradicting Computation~\ref{comp:csat}.
\end{proof}
\section{Excluding perfect balance}\label{sec:balance}
To prove Proposition~\ref{prop:nobalance}, we first show that perfect balance gives every vertex degree 15 or 16 in the same two colors, and then exclude these degree vectors.

\subsection{A common critical pair}\label{sec:allC}
Suppose $c$ is perfectly balanced. By Proposition~\ref{prop:abd} and Lemma~\ref{lem:entrance}(2), only profiles C$_1$ and C$_2$ occur, and each vertex receives tickets exactly from its two critical neighborhoods.

\begin{lemma}[Reciprocity]\label{lem:twoway}
In a perfectly balanced coloring, an edge's color is critical at one endpoint if and only if it is critical at the other.
\end{lemma}
\begin{proof}
Suppose $u$ sends a ticket to $v$ along a color-$i$ edge. Since $(v,u)$ is not an entrance, color $i$ is critical at $v$, and $v$ sends a ticket back.
\end{proof}

\begin{lemma}[Internal connectivity]\label{lem:critical-connectivity}
Every color graph in a good three-coloring of $K_n$, for $n\in\{15,16\}$, is connected.
\end{lemma}
\begin{proof}
Every color degree is at most five, since a color neighborhood carries a good two-coloring and $R(3,3)=6$. The three degrees add up to $n-1$, so each is at least $n-11>0$. Fix a color graph. Since it has no triangle, every component contains an edge whose ends have disjoint neighborhoods. Hence, each component has at least $2(n-11)$ vertices, and two components would need $4(n-11)>n$ vertices, which is impossible.
\end{proof}

\begin{lemma}[Homogeneity of critical neighborhoods]\label{lem:critical-homogeneity}
In a perfectly balanced coloring, all vertices of a critical neighborhood have the same set of critical colors.
\end{lemma}
\begin{proof}
Let $N_i(v)$ be critical. By reciprocity (Lemma~\ref{lem:twoway}), $i$ is critical at every vertex of $N_i(v)$. For $j\ne i$, the color-$j$ graph inside $N_i(v)$ is connected by Lemma~\ref{lem:critical-connectivity}. Again by reciprocity, $j$ is critical at both ends of each such edge or at neither, hence at all of $N_i(v)$ or at none.
\end{proof}

\begin{lemma}[A common critical pair]\label{lem:allC}
In a perfectly balanced coloring, one global color permutation makes every vertex have
\begin{equation}\label{eq:balanced-degrees}
d_0,d_1\in\{15,16\},\qquad d_2,d_3\le14.
\end{equation}
\end{lemma}
\begin{proof}
We take any root $r$, name its critical pair 0 and 1, put $X=N_0(r)$ and $Y=N_1(r)$, and let $x\in X$. By reciprocity, 0 is critical at $x$. Suppose the other critical color of $x$ is $j\in\{2,3\}$. Since $j$ is noncritical at $r$, reciprocity makes $j$ noncritical at every vertex of $N_j(r)$. As $j$ is critical at $x$, no color-$j$ edge joins $x$ to $N_j(r)$, and $c(xr)=0$.

The template on $X$ gives $x$ a color-$j$ neighbor in $X$, whose critical pair is $\{0,j\}$ by reciprocity. By homogeneity, all color-$j$ neighbors of $x$ have this pair, so none lies in $Y$, where 1 is critical. The color-$j$ neighbors of $x$ are thus at most five vertices of $X$ and at most five in the remaining group, where their edges avoid color $j$ and the root color of that group. Hence, $d_j(x)\le5+5=10$, while a critical color has degree at least 15. So the pair at $x$ is $\{0,1\}$, and likewise on $Y$.

Since $r$ was arbitrary, every vertex has the same critical pair as each of its critical neighbors. Any two vertices share a critical neighbor, as $31+31>60$, so all pairs agree, and one global color permutation names this pair 0 and 1.
\end{proof}

\subsection{Visible opposite neighborhoods}\label{sec:guard}
From now on, we assume only the degree vectors of Lemma~\ref{lem:allC}. At a root $r$, we write
\[
X=N_0(r),\qquad Y=N_1(r),\qquad P=N_2(r),\qquad Q=N_3(r).
\]
Since the degrees sum to 59, every vertex has profile C$_1$ or C$_2$. If necessary, we exchange colors 0 and 1, or colors 2 and 3, so that these sets have sizes $(16,16,13,14)$ at a C$_1$ root and $(15,16,14,14)$ at a C$_2$ root.

\begin{lemma}[The guard]\label{lem:guard}
Every $x\in X$ has four or five color-0 neighbors in $Y$, and every $y\in Y$ has four or five color-1 neighbors in $X$. These neighbors carry complementary paths or five-cycles in colors 2 and 3.
\end{lemma}
\begin{proof}
The 15 or 16 color-0 neighbors of $x$ are $r$ and 14 or 15 vertices outside $X$. Inside each of $Y$, $P$ and $Q$, their edges avoid color 0 and the root color of the group, so each group contains at most five of them, hence four or five. As in the proof of Lemma~\ref{lem:marked-shapes}, they carry complementary paths or five-cycles. For $y$, we exchange colors 0 and 1.
\end{proof}

By Theorem~\ref{thm:classification} and Lemma~\ref{lem:templatesym}(c), we may relabel the vertices of $X$ and $Y$, without changing any edge color, so that they carry templates $V_X,V_Y\in\{\U,\T,\U^-,\T^-\}$ with palettes $(1,2,3)$ and $(0,2,3)$.

For $x\in X$ and $y\in Y$, we call
\[
\Omega_x=(X\cup Y)\cap N_1(x),\qquad
\Omega_y=(X\cup Y)\cap N_0(y)
\]
their \emph{visible opposite neighborhoods}. Since $N_1(x)$ and $N_0(y)$ have 15 or 16 vertices, each is a copy of $\U$, $\T$, $\U^-$ or $\T^-$. By restriction, the colored graphs on $\Omega_x$ and $\Omega_y$ embed, with all colors preserved, into $\U$ or $\T$ with palettes $(0,2,3)$ and $(1,2,3)$.

The templates on $X$ and $Y$ fix the parts $A_x=X\cap N_1(x)$ of $\Omega_x$ and $B_y=Y\cap N_0(y)$ of $\Omega_y$. Both parts have four or five vertices (see the proof of Lemma~\ref{lem:marked-shapes}). Their edges avoid colors 0 and 1, so a five-vertex part carries a color-2 pentagon and its color-3 complement. We call these sets the \emph{anchors} and the other visible vertices \emph{members}. We impose embedding conditions only for visible neighborhoods with five-vertex anchors. A vertex $y\in Y$ is a member of $\Omega_x$ exactly when $c(xy)=1$, and $x\in X$ is a member of $\Omega_y$ exactly when $c(xy)=0$ (Figure~\ref{fig:opposite-grid}). A visible neighborhood with a five-vertex anchor has at most 11 members, since its critical neighborhood has at most 16 vertices.

\begin{figure}[!htbp]
\centering
\definecolor{braceblue}{HTML}{0072B2}
\definecolor{braceorange}{HTML}{E69F00}
\definecolor{bracegreen}{HTML}{009E73}
\begin{tikzpicture}[x=1cm,y=1cm,
  note/.style={font=\small,align=left,text width=6.3cm,inner sep=2pt},
  lab/.style={font=\small,inner sep=1pt},
  tick/.style={line width=1.2pt}]
  \def\s{.22}
  \newcommand{\panel}[2]{
    \begin{scope}[shift={(#1,0)}]
      \fill[black!7] (0,{10*\s}) rectangle ({16*\s},{11*\s});
      \fill[black!7] ({10*\s},0) rectangle ({11*\s},{16*\s});
      \draw[step=\s,draw=black!30,line width=.3pt] (0,0) grid ({16*\s},{16*\s});
      \draw[draw=black!70,line width=.6pt] (0,0) rectangle ({16*\s},{16*\s});
      \draw[draw=braceorange,line width=1.2pt,fill=braceorange!40] ({10*\s},{10*\s}) rectangle ({11*\s},{11*\s});
      \node[lab,anchor=east] at (-.1,{10.5*\s}) {$x_i$};
      \node[lab,anchor=south] at ({10.5*\s},{16*\s+.05}) {$y_j$};
      \node[font=\small,anchor=south] at ({8*\s},{16*\s+.5}) {#2};
    \end{scope}}
  \panel{0}{(a) $c(x_iy_j)=1$: $y_j\in\Omega_{x_i}$}
  \foreach \r in {2,4,7,9,12} {
    \draw[draw=braceblue,line width=1.2pt] ({10*\s},{(16-\r)*\s}) rectangle ({11*\s},{(17-\r)*\s});
  }
  \node[note,anchor=north] at ({8*\s},-.3)
    {The five boxed entries lie in column $j$, in the rows of the anchor $A_{x_i}$. Their colors, from $\{0,2,3\}$, are those of the position of $y_j$ in the table.};
  \def\ox{7.3}
  \panel{\ox}{(b) $c(x_iy_j)=0$: $x_i\in\Omega_{y_j}$}
  \foreach \c in {2,5,8,13,16} {
    \draw[draw=bracegreen,line width=1.2pt] ({\ox+(\c-1)*\s},{10*\s}) rectangle ({\ox+\c*\s},{11*\s});
  }
  \node[note,anchor=north] at ({\ox+8*\s},-.3)
    {The five boxed entries lie in row $i$, in the columns of the anchor $B_{y_j}$. Their colors, from $\{1,2,3\}$, are those of the position of $x_i$ in the table.};
\end{tikzpicture}
\caption{A cross edge $x_iy_j$ (orange) and the embedding conditions of $\Omega_{x_i}$ and $\Omega_{y_j}$. In (a), it has color 1, so $y_j$ is a member of $\Omega_{x_i}$ and takes a nonanchor position in a table of Section~\ref{sec:anchor-tables}. In (b), it has color 0, so $x_i$ is a member of $\Omega_{y_j}$. At a C$_1$ root, all 32 conditions share this grid of 256 edges. Each entry decides membership in the conditions of its row and column and is an anchor entry in up to ten others. Box colors are not edge colors.}
\label{fig:opposite-grid}
\end{figure}
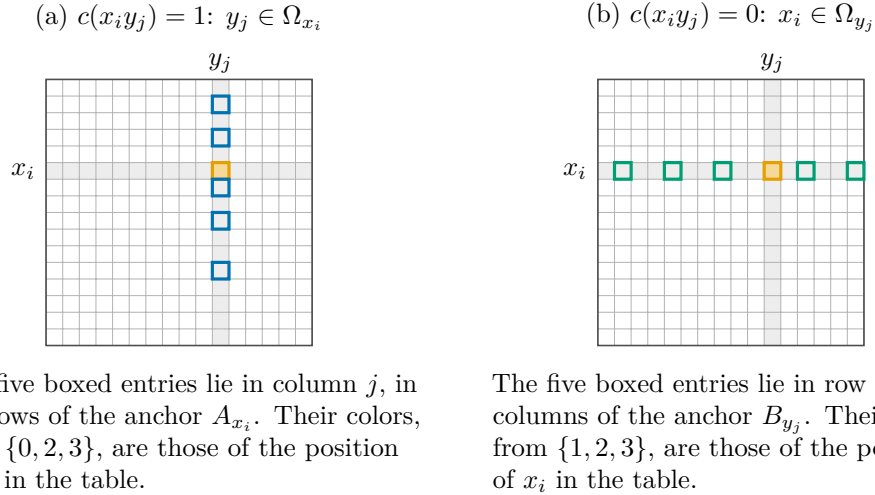

\subsection{Finite template tables}\label{sec:anchor-tables}\label{sec:completions}
A \emph{table} is a relabeling of $\U$ or $\T$ onto positions $0,\ldots,15$ whose \emph{anchor positions} $0,\ldots,4$ carry the reference pentagon, with template color 1 on $01,12,23,34,40$ and color 2 on the other pairs. The other 11 positions are the \emph{nonanchor positions}. An \emph{ordered anchor} is a sequence of five vertices carrying this reference coloring.

In a template, each five-subset independent in color 0 carries complementary pentagons in colors 1 and 2. Its ten orders along the color-1 pentagon are ordered anchors, and every ordered anchor arises in this way, since it has no color-0 edge.

\begin{computation}[The anchored table library]\label{comp:charts}
Each of $\U$ and $\T$ has 16 five-subsets independent in color 0. Six representative tables, one from $\U$ and five from $\T$, cover the resulting $2\cdot16\cdot10=320$ ordered anchors. For every ordered anchor, there is a bijection of the 16 positions that preserves all 120 edge colors and sends the ordered anchor, in order, to the anchor positions of its representative.
\end{computation}
We order each five-vertex $A_x$ and $B_y$ along its color-2 pentagon, which has template color 1 in both palettes $(0,2,3)$ and $(1,2,3)$, so the library covers all ten such orders.

\begin{proposition}[Exhaustive anchored tables]\label{prop:anchored}
Let $G$ be a three-colored graph with an ordered anchor colored as above. Then $G$ embeds in $\U$ or $\T$ if and only if we can assign its other vertices distinct nonanchor positions in one of the six tables, preserving the colors of all edges at these vertices.
\end{proposition}
\begin{proof}
An embedding sends the ordered anchor to one of the 320 enumerated ones, and composing it with that anchor's bijection gives a table assignment. Conversely, a table assignment composed with the inverse bijection is a color-preserving embedding into $\U$ or $\T$.
\end{proof}

\subsection{The compatibility formulas}\label{sec:overlap}
For each pair $(V_X,V_Y)$, the \emph{compatibility formula} $\Psi_{V_X,V_Y}$ has the edge variables and goodness clauses of Definition~\ref{def:case}. Its unit clauses fix the 59 root edges and the 240 or 225 template edges in $X$ and $Y$. Cardinality clauses impose $d_0,d_1\le16$ and $d_2,d_3\le14$ everywhere. Since the degrees sum to 59, these bounds are equivalent to~\eqref{eq:balanced-degrees}. For each $x\in X$, auxiliary variables encode the guard and, if $|A_x|=5$, the table criterion of Proposition~\ref{prop:anchored} through three kinds of choices:
\begin{enumerate}
\item Guard selectors choose exactly one four- or five-subset of $V_Y$ that carries complementary paths or five-cycles in colors 2 and 3. The chosen subset must equal $\{y\in Y:c(xy)=0\}$.
\item Six table selectors choose exactly one of the six anchored tables.
\item For each $y\in Y$ and nonanchor position $t$, a position variable records whether $y$ occupies $t$. A vertex receives exactly one position when $c(xy)=1$ and none otherwise. Distinct members receive distinct positions. Every anchor-to-member color agrees with the selected table, read with palette $(0,2,3)$, and so does every member-to-member color, which $V_Y$ fixes.
\end{enumerate}
We impose the same conditions for every $y\in Y$, with $X,Y$ and colors 0,1 exchanged, and write them as exactly-one and implication clauses. Each visible neighborhood has its own auxiliary variables, but all share the edge variables. The CNF files of the reproduction package share the table-independent constraints (see Section~\ref{sec:verification}).

Exchanging colors 0 and 1 everywhere, and with them $X$ and $Y$, carries $\T\U$ to $\U\T$ and preserves the degree vectors and all embedding requirements. So at a C$_1$ root, we need only $\U\U$, $\U\T$ and $\T\T$. At a C$_2$ root, $X$ has 15 vertices and $Y$ has 16, so we need $\U^-\U$, $\U^-\T$, $\T^-\U$ and $\T^-\T$.

\begin{lemma}[From a coloring to a compatibility formula]\label{lem:overlap-sound}
A good coloring satisfying~\eqref{eq:balanced-degrees} everywhere gives, at any root, a satisfying assignment to one of the compatibility formulas needed at that root.
\end{lemma}
\begin{proof}
We normalize the root partition and the templates on $X$ and $Y$, using the exchange above if necessary to obtain a listed template pair. We set $x_{e,i}$ true exactly when $e$ has color $i$, so goodness, the fixed unit clauses and the cardinality clauses hold.

For $x\in X$, we select $N_0(x)\cap Y$, one of the guard choices by Lemma~\ref{lem:guard}. When $|A_x|=5$, classification of $N_1(x)$ and Proposition~\ref{prop:anchored} give a table and an injective placement of $\Omega_x$ that sends its ordered anchor to positions $0,\ldots,4$. We select that table and set a position variable true exactly when its vertex is a member placed at that position. This way, every member has one position and every nonmember none, and the placement preserves all anchor-to-member and member-to-member colors.

Similarly, we treat every $y\in Y$ with colors and groups exchanged. The auxiliary variables of different vertices are disjoint and every edge has one color, so all conditions hold at once.
\end{proof}

For each of $\Psi_{\U,\U}$, $\Psi_{\U,\T}$ and $\Psi_{\T,\T}$, we form an \emph{all-five formula} by removing the cardinality clauses and forbidding all four-vertex guard choices.

\begin{computation}[The all-five exclusions]\label{comp:fiveguards}
The three all-five formulas are unsatisfiable. Kissat and CaDiCaL, run as in Computation~\ref{comp:abdsat}, refuted each of them.
\end{computation}

\begin{computation}[Four compatibility exclusions]\label{comp:overlap}
The four C$_2$ compatibility formulas are unsatisfiable. Kissat and CaDiCaL, run as in Computation~\ref{comp:abdsat}, refuted each of them.
\end{computation}

\begin{proposition}\label{prop:noallC}
No good four-coloring of $K_{60}$ satisfies~\eqref{eq:balanced-degrees} at every vertex.
\end{proposition}
\begin{proof}
If every vertex of such a coloring were C$_1$, then $d_0=d_1=16$ everywhere, and the proof of Lemma~\ref{lem:guard} would give every set $N_0(x)\cap Y$ and $N_1(y)\cap X$ five vertices. The assignment of Lemma~\ref{lem:overlap-sound} would then satisfy an all-five formula, contradicting Computation~\ref{comp:fiveguards}. Otherwise, by Lemma~\ref{lem:overlap-sound}, such a coloring would give, at a C$_2$ root, a satisfying assignment to a formula excluded by Computation~\ref{comp:overlap}.
\end{proof}
\begin{proof}[Proof of Proposition~\ref{prop:nobalance}]
By Lemma~\ref{lem:allC}, a perfectly balanced coloring has, after a color permutation, the degree vectors excluded by Proposition~\ref{prop:noallC}.
\end{proof}
\section{Computations and verification}\label{sec:verification}
We have reduced Theorem~\ref{thm:main} to the unsatisfiability of 296,074 marked-case formulas, three all-five formulas and four compatibility formulas, 296,081 in all. The generators write them on the same edge variables of $K_{60}$. We ran Kissat~\cite{Kissat} and CaDiCaL~\cite{Cadical} on each, retrying timed-out or interrupted runs, and both solvers returned UNSAT for every formula.

For our computation, we used an Apple M4 Max with 128 GiB of memory and ran up to 12 solver processes at once. The 296,081 successful UNSAT runs per solver took approximately 872 CPU hours for Kissat and 1,536 for CaDiCaL (Table~\ref{tab:solver-costs}). These times exclude previous unsuccessful attempts, input generation, finite enumeration and Lean compilation.

\begin{table}[htbp]
\centering
\small
\caption{CPU hours for successful UNSAT runs, once per formula. In a name such as D/TTU, the letters are the template types of $R_1,R_2,R_3$.}\label{tab:solver-costs}
\begin{tabular}{@{}lrrr@{}}
\toprule
Case family & Formulas & Kissat & CaDiCaL\\
\midrule
A & 360 & 0.27 & 1.88\\
B & 5{,}096 & 6.45 & 37.56\\
D/UUU & 679 & 1.63 & 2.65\\
D/UUT & 679 & 1.68 & 5.53\\
D/UTU & 6{,}576 & 17.60 & 36.12\\
D/UTT & 6{,}576 & 18.03 & 40.67\\
D/TTU & 18{,}204 & 59.09 & 155.32\\
D/TTT & 18{,}204 & 59.95 & 159.63\\
E/UUU & 2{,}037 & 4.77 & 6.29\\
E/UUT & 6{,}111 & 14.51 & 22.56\\
E/UTT & 59{,}184 & 154.69 & 211.06\\
E/TTT & 163{,}836 & 518.39 & 767.98\\
C$_1$ entrance & 692 & 0.46 & 1.32\\
C$_2$ entrance & 7{,}840 & 12.51 & 85.34\\
All-five formulas & 3 & 0.03 & 0.03\\
C$_2$ compatibility & 4 & 1.48 & 1.60\\
\midrule
Total & 296{,}081 & 871.53 & 1{,}535.55\\
\bottomrule
\end{tabular}
\end{table}

To verify the computation, we used two methods. First, separate programs matched the solver inputs one to one with the formulas emitted by Lean. They compared the ordered signed facts, the numbers of variables, clauses and bytes, and the hash of the full DIMACS file. For the marked cases, the counts and hashes were derived from the unit clauses and the common base~\eqref{eq:baseclauses}, which Lean emitted once. The all-five and compatibility formulas of Section~\ref{sec:overlap} were emitted in full. Second, we checked the input file, solver binary, command, exit status and output of every run.

The final theorem of our Lean~4 formalization~\cite{Lean,mathlib} is \texttt{UniformK60.ramsey60\_of\_exact\_unsat}. Its premise asserts unsatisfiability of the exact formulas used in the formalization. Otherwise, it uses only the standard axioms \texttt{propext}, \texttt{Classical.choice} and \texttt{Quot.sound}. Lean checks the classification certificates, but not the SAT exclusions.

We formalize the degree profiles, the double counting of tickets, the template classification with its 55 CNF/LRAT certificate pairs and the marked-neighborhood covers, where the C covers assume no template on the noncritical groups. We also formalize the perfect-balance argument, the anchored tables and the passage from a coloring to a satisfying assignment in Lemmas~\ref{lem:soundness} and~\ref{lem:overlap-sound}, where the second assigns all auxiliary variables at once. With exhaustive finite checks and our proofs about symmetries and pullbacks, we establish the template automorphisms and palette maps, the 600 words of each shape and the orbits of supports and words. We also check the 320 bijections of Computation~\ref{comp:charts} and that every retained case list satisfies the finite coverage condition.
\section{Data and AI assistance}\label{sec:data}
The reproduction package~\cite{repository} contains the Lean sources and manifests, the template data, finite tables, case lists and symmetry witnesses, the formula generators and the verification programs. Its executable instructions rebuild the Lean sources, regenerate and compare the inputs and rerun both solvers. It also records a completed build, every comparison and measurements for the verification steps of Section~\ref{sec:verification}. For every solver run, it keeps the command, resource use and full output.

AI tools assisted extensively with the mathematics, counterexample analysis, implementation, experiments and formalization. GPT~5.6 in Codex (ultra mode) and Fable~5 in Claude Code (ultracode mode) contributed to the A/B/D reductions. GPT-6 Pro provided mathematical proposals and analysis and helped prepare and edit the manuscript. GPT-6 Astra in Codex (ultra mode) contributed the C reductions, their integration and the simplified compatibility argument. The author directed the research and computations and takes responsibility for the result, its exposition, computational artifacts, verification claims and any necessary corrections.

\end{document}